\documentclass[11pt]{amsart}
\usepackage{amsmath,amsthm,amscd,amsfonts,amssymb,epic,eepic,bbm, graphicx, youngtab,txfonts, ytableau}
\usepackage{tikz,color,graphicx,tikz-cd,xcolor,tkz-euclide}
\usepackage{enumerate}
\usepackage{array}
\usepackage{mathtools}
\usepackage{url}

\usepackage{kotex-logo}

\usepackage[T1]{fontenc}
\PassOptionsToPackage{no-math}{fontspec}
\usepackage{iftex}

\ifPDFTeX
\usepackage[finemath]{kotex}
\usepackage{dhucs-nanumfont}
\fi

\ifXeTeX
\usepackage[unfonts]{kotex}
\xetexkofontregime{latin}[puncts=prevfont, colons=prevfont, cjksymbols=hangul]

\defaultfontfeatures+{
Ligatures          = TeX,
ItalicFont         = *,
ItalicFeatures     = {FakeSlant=.17}, 
BoldItalicFeatures = {FakeSlant=.17},
}

\setmainhangulfont{NanumMyeongjo}
\setsanshangulfont{NanumGothic}

\fi

\numberwithin{equation}{section}
\allowdisplaybreaks

\catcode`\@=11
\@namedef{subjclassname@2020}{%
  \textup{2020} Mathematics Subject Classification}
\catcode`\@=12

\newtheorem{thm}{Theorem}[section]
\newtheorem{cor}[thm]{Corollary}

\newtheorem{lem}[thm]{Lemma}

\newtheorem{prop}[thm]{Proposition}
\newtheorem{rem}[thm]{Remark}

\theoremstyle{definition}
\newtheorem{defn}[thm]{Definition}

\theoremstyle{remark}

\newtheorem{example}[thm]{Example}

\DeclareMathOperator\prmx{fdes}
\DeclareMathOperator\rank{rank}

\DeclareMathOperator\height{ht}

\DeclareMathOperator\ud{ud}
\DeclareMathOperator\dist{dist}

\def\Z{\mathbb{Z}}

\newcommand\set[1]{\left\{#1\right\}}
\newcommand\deli{:}
\def\IS{\mathcal{IS}}

\newcommand\LFp[1]{\mathcal{L}\mathcal{F}_{#1}}

\newcommand\Hw[1]{\mathcal{H}_{#1}}
\newcommand\wHw[1]{w\mathcal{H}_{#1}}

\newcommand\uline[1]{{#1}}

\newcommand\wtd[1]{{#1}}
\newcommand\pex[1]{^{(#1)}}

\title
{On (102,000)-avoiding inversion sequences}

\author{Sangwook Kim$^\dag$}
\address[Sangwook Kim]{Department of Mathematics, Chonnam National University, Gwangju, 61186, South Korea}
\email{swkim.math@chonnam.ac.kr}
\thanks{\dag Corresponding author}

\author{Seunghyun Seo}
\address[Seunghyun Seo]{Department of Mathematics Education, Kangwon National University, Chuncheon, 24341, South Korea}
\email{shyunseo@kangwon.ac.kr}
\thanks{}

\author{Heesung Shin}
\address[Heesung Shin]{Department of Mathematics, Inha University, 100 Inharo, Michuhol, Incheon, 22212, South Korea}
\email{shin@inha.ac.kr}
\thanks{}

\keywords{
inversion sequences, 
pattern avoidance,
simple $H$-paths}
\subjclass[2020]{Primary 05A19; Secondary 05A05, 05A15}

\begin{document}

\begin{abstract}
   In this article, we study $(102,000)$-avoiding inversion sequences with a fixed number of distinct elements.
   By introducing simple $H$-paths, we derive the trivariate generating function for these inversion sequences with respect to their length, number of distinct elements, and rank.
   As consequences, we obtain an explicit formula for the number of $(102,000)$-avoiding inversion sequences with fixed length and number of distinct elements
   and we also provide a formula for those with fixed number of distinct elements and rank.
   In particular, we show that both the number of $(102,000)$-avoiding inversion sequences with a fixed number of distinct elements whose maximum element occurs exactly once
   and the number of those whose rank is zero are given by the $3$-Fuss-Catalan numbers.
\end{abstract}

\maketitle
    

\section{Introduction}

An \emph{inversion sequence of length \( n \)} is an integer sequence \( e=(e_1, e_2, \dots, e_n) \) satisfying 
\( 0 \leq e_i <i \) for all \( i \in [n] \).
While pattern avoidance in permutations has been extensively studied for more than fifty years (see Kitaev~\cite{Kitaev11} for comprehensive survey), 
the study of pattern avoidance in inversion sequences is a more recent development,
initiated independently around 2015 by Corteel et al.~\cite{CMSW16}, and Mansour and Shattuck~\cite{MS15}.
These pioneering works investigated inversion sequences avoiding patterns of length three.
Subsequently, Martinez and Savage~\cite{MartinezSavage18} reformulated the notion of length-three 
pattern from a word of length three to a triple of binary relations.
Yan and Lin~\cite{YanLin20} later conducted a systematic study of inversion sequences avoiding pairs of length-three patterns, resolving 78 cases and identifying 32 open ones.
Kotsireas, Mansour, and Y{\i}ld{\i}r{\i}m~\cite{KMY24} enumerated eight of these open cases using generating trees and the kernel method, 
and one more case was solved independently by Chen and Lin~\cite{ChenLin24} and Pantone~\cite{Pantone24}.
Most recently, Testart~\cite{Testart24} completed the enumeration of inversion sequences avoiding one or two patterns of length three. 
In particular, Testart~\cite{Testart24} provided minimal polynomials for the generating functions for four cases, 
including that of $(102,000)$-avoiding inversion sequences.

In some cases, statistics on inversion sequences refine the enumerative results for pattern-avoiding inversion sequences and reveal connections with statistics on other combinatorial families.
Such statistics include the number of distinct element, the number of zeros, the number of repeats, the number of maximal elements, and the maximum value 
(see~\cite{CJL19, CMSW16, MS15, MartinezSavage18, YanLin20}). 
In this article, we focus on the statistic $\dist(e)$, the number of distinct elements in an inversion sequence $e$.
Several studies have investigated the enumeration of pattern-avoiding inversion sequences with a given length and a fixed number of distinct elements.
Corteel et al.~\cite{CMSW16} derived a recurrence relation for the case of $000$-avoiding inversion sequences.
Martinez and Savage~\cite{MartinezSavage18} studied the case avoiding $(100, 011)$,
while the cases avoiding $(000, 101)$ and $(000,110)$ were conjectured by Martinez and Savage~\cite{MartinezSavage18} and later confirmed by Cao, Jin and Lin~\cite{CJL19}.

In~\cite{HKSS24}, Huh et al. studied three statistics on $(101,102)$-avoiding and 
$(101, 021)$-avoiding inversion sequences, respectively.
They derived enumerative results for these pattern-avoiding inversion sequences with respect to considered statistics 
and established correspondences with statistics on other combinatorial objects including 
Schr\"{o}der paths without triple descents, $(2341, 2431, 3241)$-avoiding permutations, and weighted ordered trees
by introducing the notion of $F$-paths.
Subsequently, Huh et al.~\cite{HKSS25} introduced a statistics called \emph{rank} on $102$-avoiding inversion sequences, extending one of the statistic defined for $(101, 102)$-avoiding inversion sequences  in~\cite{HKSS24}.
They constructed a bijection between the set of $102$-avoiding inversion sequences with a fixed rank and the set of $2$-Schr\"{o}der paths having neither peaks nor valleys that end with a diagonal step and possessing a fixed number of blocks by introducing labeled $F$-paths.
They also enumerated the doubly-avoiding cases where the second pattern $\tau \in \{ 001, 011, 012, 021, 110, 120, 201, 210 \}$ is also avoided. 

In this article, we study \( (102, 000) \)-avoiding inversion sequences with a fixed number of distinct elements.
We introduce \emph{weighted $H$-walks}, a variant of labeled $F$-paths, and show that they are in bijection with $102$-avoiding inversion sequences.
For $(102,000)$-avoiding inversion sequences, these weighted $H$-walks specialize to a class of paths that we call \emph{simple $H$-paths}.
Using simple $H$-paths, we derive the trivariate generating function for the number of $(102,000)$-avoiding inversion sequences $e$ with respect to their length, $\dist (e)$ and $\rank (e)$.
As a consequence, we obtain an explicit formula for the number of $(102,000)$-avoiding inversion sequences of a given length and with a fixed number of distinct elements.
Furthermore, we show that the number of $(102,000)$-avoiding inversion sequences whose number of distinct elements is $m$ and whose maximum element occurs only once is given by the $3$-Fuss-Catalan number (A002293 in OEIS~\cite{OEIS}).
The number of $(102,000)$-avoiding inversion sequences $e$ with fixed $\dist(e)$ and $\rank (e)$ corresponds to OEIS sequence A355174~\cite{OEIS}.
In particular, the number of $(102,000)$-avoiding inversion sequences $e$ with fixed $\dist(e)$ and $\rank(e) = 0$ is counted by the $3$-Fuss-Catalan number.
The number of $(102,000)$-avoiding inversion sequences $e$ with fixed $\dist(e)$ is enumerated by OEIS sequence A069271.

The rest of the paper is organized as follows.
Section~\ref{sec:preliminaries} introduces the necessary definitions and background.
In Sections~\ref{sec:102-000-avoiding} and~~\ref{sec-refined-enumeration}, we study $(102,000)$-avoiding inversion sequences in detail.
In Section~\ref{sec:102-000-avoiding}, we enumerate the cardinality of the set of simple $H$-paths ending with a north step 
and derive a formula for the number of $(102,000)$-avoiding inversion sequences whose number of distinct elements is $m$ and whose maximum element occurs exactly once.
In Section~\ref{sec-refined-enumeration}, we consider the trivariate generating function for the number of $(102,000)$-avoiding inversion sequences $e$ with respect to their length, 
number of distinct elements, and rank, and we obtain explicit formulas for the cases with fixed $\dist(e)$ and with fixed $(\dist(e), \rank(e))$.

\section{Preliminaries}
\label{sec:preliminaries}

In this section, we provide definitions for the objects we discuss in this article.

\subsection{Inversion sequences}
An integer sequence $e=(e_1, e_2, \dots, e_n)$ is called an \emph{inversion sequence of length $n$}
if $0 \leq e_j \leq j-1$ for all $j\in [n]:=\{ 1,2,\dots,n \}$.

Pattern avoidance in inversion sequences can be defined in a similar way 
to that in permutations.
Given a word \( w \in \set{0, 1, \dots, k-1}^k\),
let the word obtained by replacing the $i$-th smallest entry in $w$ with $i-1$
be called the \emph{reduction} of $w$.
We say that an inversion sequence \( e=(e_1, e_2, \dots, e_n) \) 
\emph{contains} the pattern $w$
if there exist some indices $i_1 < i_2 < \dots < i_k$ 
such that the reduction of $e_{i_1} e_{i_2} \dots e_{i_k}$ is $w$.
Otherwise, $e$ is said to \emph{avoid} the pattern $w$.
Let $\IS_n$ denote the set of inversion sequences of length $n$ and 
$\IS_n(w_1, \dots, w_r)$ denote the set of inversion sequences of length $n$ avoiding all the patterns $w_1, \dots, w_r$.

We denote the largest 
integer in $e$ by $\max(e)$.
We also denote by \( \prmx(e) \) 
the position $p$ of the first descent, i.e.,
\begin{align*}
  e_1 \leq e_2\leq \dots \leq e_p > e_{p+1}
\end{align*}
with $e_{n+1} = -1$.
For 
$e=(e_1, e_2, \dots, e_n) \in \IS_n(102)$,
if $p \!= \prmx(e)$,
then 
$e_p = \max(e)$.
In this case, we define 
\begin{align*}
\rank(e) 
:= \prmx(e) - \max(e) - 1
= (p - 1) - e_p \geq 0.
\end{align*}
Let $\IS_{n,t}(102, w)$ denote the set of inversion sequences $e \in \IS_n (102, w)$ with $\rank (e) = t$.

\subsection{Labeled $F$-paths}

In \cite{HKSS24}, Huh et al.\ introduced \emph{F-paths} and constructed a bijection between 
$(102, 101)$-avoiding
inversion sequences and $F$-paths. 
An \emph{$F$-path of length $\ell$} is a lattice path in $\Z^2$
that starts at the origin and does not go below the line $y=x$
as a sequence of lattice points
$$\left((x_0, y_0), (x_1, y_1), \dots, (x_{\ell}, y_{\ell})\right)$$
of which every step $(x_{j} - x_{j-1}, y_{j}-y_{j-1})$, denoted by $s_j$, 
is 
in the set
$$F:=
\{(0,1)\} \cup \{(a,b) \deli a\geq 1,~b\leq 1\},
$$
for $j=1, 2, \ldots, \ell$.

In \cite{HKSS25}, Huh et al.\ extended the family of $F$-paths
by assigning a label to each step, in order to accommodate
the superset $\IS_n(102)$ of $\IS_n(102,101)$.
A \emph{labeled $F$-path} is an $F$-path 
where every 
step $(a, 1)$ is
assigned a label $(a; 1)$ and 
every other step $(a, b)$ with $b\leq 0$ is assigned a label 
$(a; b_1, \dots, b_k)$
for some nonpositive integers $b_1, \dots, b_k$ with $k \ge 1$ such that $b_1 + \cdots + b_k = b$.
We say that a step with a label $(a; b_1, \dots, b_k)$ has the \emph{semilength} $k$.
Define the \emph{semilength} of a labeled $F$-path by the sum of the semilengths of its steps.
In Figure~\ref{fig:F},
the labeled $F$-path $Q$ has exactly $20$ steps, 
but the semilength of $Q$ is equal to $26$, 
since the semilengths of two steps before the last step $(0,1)$ of $Q$ are $3$ and $5$.
Let $\LFp{n}$ be the set of labeled $F$-paths of semilength $n$.

Given a labeled $F$-path 
$Q = \left((0, 0), (x_1, y_1), \dots, (x_\ell, y_\ell)\right)$,
define the \emph{height} of $Q$
by the value $y_\ell-x_\ell$ of the last lattice point $(x_\ell, y_\ell)$ of $Q$
and denote it by $\height(Q)$.
Let $\LFp{n,t}$ be the set of $Q$ in $\LFp{n}$ with $\height(Q) = t$.

\begin{figure}[t]
\[
  \begin{tikzpicture}[scale=0.48]
    \pgfmathsetmacro{\n}{16}
    \pgfmathsetmacro{\m}{\n+1}
    \draw[gray!30, ultra thin] (0,0) grid (\n,\n);
    \draw [-stealth] (0,0) -- (\m,0);
    \draw [-stealth] (0,0) -- (0,\m);
    \draw [dashed] (0,0) -- (\n,\n);

    \coordinate (0) at (0,0);
    \coordinate (1) at (0,1);
    \coordinate (2) at (1,2);
    \coordinate (3) at (2,3);
    \coordinate (4) at (4,4);
    \coordinate (5) at (4,5);
    \coordinate (6) at (4,6);
    \coordinate (7) at (4,7);
    \coordinate (8) at (4,8);
    \coordinate (9) at (4,9);
    \coordinate (10) at (4,10);
    \coordinate (b) at (3.9,9);
    \coordinate (11) at (6,9);
    \coordinate (c1) at (4.5,12.5);
    \coordinate (c2) at (5.6,12);
    \coordinate (c3) at (5.6,11);
    \coordinate (c4) at (5.6,10.6);
    \coordinate (12) at (7,10);
    \coordinate (13) at (7,11);
    \coordinate (14) at (7,12);
    \coordinate (15) at (7,13);
    \coordinate (16) at (7,14);
    \coordinate (17) at (7,15);
    \coordinate (d) at (7.5,15.6);
    \coordinate (18) at (8,15);
    \coordinate (e) at (8.3,14.2);
    \coordinate (19) at (9,12);
    \coordinate (20) at (9,13);

    \foreach \i in {0,...,20}{
      \filldraw (\i) circle (4pt);
    }

    \foreach \i in {1,...,20}{
    \pgfmathsetmacro{\j}{\i-1}
    \draw[ultra thick] (\j) -- (\i);
    }

    \node at (0, 0) [below left] {\footnotesize $O$};
    \foreach \i in {5,10,...,17}{
      \node at (\i, 0) [below] {\footnotesize$\i$};
      \node at (0,\i) [left] {\footnotesize$\i$};
    }
    \node at (d) [] {\footnotesize$(1;0,0,0)$};
    \node at (e) [right] {\footnotesize$(1;-1, 0, 0, -1, -1)$};
  \end{tikzpicture}
\]
\caption{
  A labeled $F$-path $Q$ of semilength $26$ and height $4$,
  with the labels $(a;b)$ are omitted for brevity. 
}
\label{fig:F}
\end{figure}
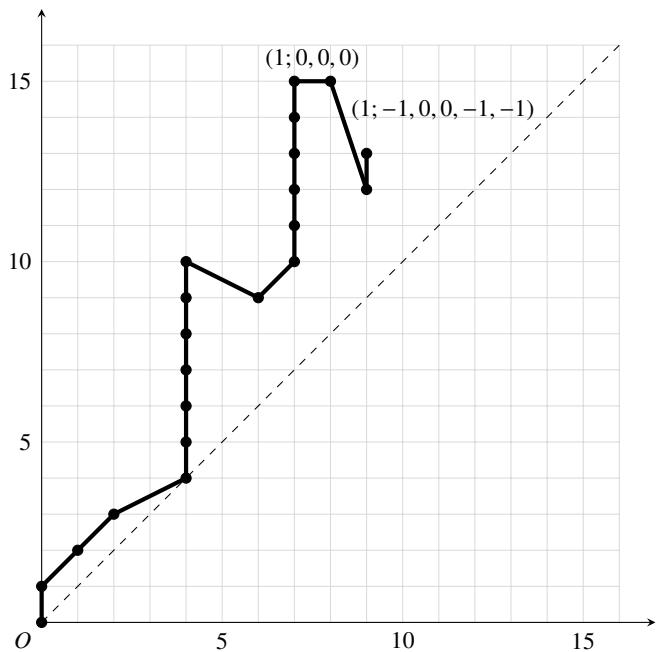

\subsection{Weighted {\it H}-walks}\label{sec:WHW}

Recall $F = \{(0,1)\} \cup \{(a,b) \deli a\geq 1,~b\leq 1\}$.
Let $H$ be the set defined by 
$$
H \coloneqq 
F\cup \{ (0,b)\deli b\leq -1\}.
$$
We define an \emph{\( H \)-walk} $R$ of length \( \ell \) to be a lattice walk 
$ R=((0,0), (x_1,y_1), \dots,  (x_{\ell},y_{\ell}))$
with steps in the set $H$,
where $R$ satisfies that
\begin{enumerate}
\item \( x_i \leq y_i\) for each \( i=1,2,\dots, \ell \) and
\item if \( s_i=(0,1)\) then \( s_{i+1}\notin \{ (0,b)\deli b\leq -1\}\) for each \( i=1,\dots, \ell-1 \). 
\end{enumerate}
For a negative integer $b$, we call a step \( (0,b) \) a \emph{south step}. 
Given an $H$-walk \( R =((0,0), (x_1,y_1), \dots,  (x_{\ell},y_{\ell})) \), the height of $R$ is defined to be
\( y_{\ell}-x_{\ell} \), which is denoted by \( \height(R) \).

Now we define weighted \( H \)-walks.
\begin{defn}
A \emph{weighted \( H \)-walk} is an \( H \)-walk, where each south step \( (0,b) \) is weighted by 
a positive integer $r$. We simply denote by \( r(0,b) \) a south step \( (0,b) \) of weight \(r \). For $r=1$, we simply write $(0,b)$ instead of $1(0,b)$.
For convention, we define the weight of a step \( (a,b) \in F \) by $1$. 
The \emph{semilength} of a weighted \( H \)-walk is defined to be the sum of the weight of steps. 
\end{defn}
For nonnegative integers $n$ and $t$, let \(\wHw{n} \) denote the set of weighted \( H \)-walks of semilength \( n \) and \(\wHw{n,t} \) denote the set of $R$ in $\wHw{n}$ with $\height(R)=t$.

Now we give a bijection between the sets $\LFp{n,t}$ and $\wHw{n,t}$.
Consider a labeled $F$-path $Q$. Suppose $Q$ has a step $s=(a,b)$ with a label $(a;b_1, b_2, \ldots, b_k)$ of semilength $k\ge 2$. From now on, we call such a step a \emph{long step}. 
For the long step $s$, let 
$$I(s):=\{i \deli 2\leq i \leq k-1,~b_i\ne0\}.$$
We substitute the long step $s$ in $Q$ with the sequence $S$ of (weighted) steps as follows:
\begin{enumerate}
\item If $I(s)= \emptyset$, then $S=(a, b_k+1)~\uline{(k-1)}(0,b_1-1)$. 
\item If $I(s)=\{{i_1},\ldots,{i_p}\}_<$ for some positive integer $p$, then 
$$S=(a,b_k+1)~\uline{(k-i_p)}(0,b_{i_p})~\uline{(i_{p}-i_{p-1})}(0, b_{i_{p-1}})
\cdots \uline{(i_2-i_1)}(0,b_{i_1})~ \uline{(i_1-1)}(0,b_1-1).$$
\end{enumerate}
By substituting all long steps $s$ in $Q$ with $S$, we get a weighted $H$-walk. 
In fact, this induces a {bijection} from $\LFp{n,t}$ to $\wHw{n,t}$. 
From now on, we refer to this bijection as $\eta:\LFp{n,t}\to\wHw{n,t}$. 

\begin{example}
Consider the labeled $F$-path $Q$ in Figure~\ref{fig:F}. There are two long steps in $Q$.
\begin{itemize}
\item The long step $(1,0)$ with the label $(1;0,0,0)$ is replaced by $(1,1)~2(0,-1)$.
\item The long step $(1,-3)$ with the label $(1;-1,0,0,-1,-1)$ is replaced by $(1,0)~(0,-1)~3(0,-2)$.
\end{itemize}
Thus the weighted $H$-walk $\eta(Q)$ is given by
$$
\eta(Q)=(0,1)~(1,1)^2~(2,1)~(0,1)^6~(2,-1)~(1,1)~(0,1)^5~(1,1)~2(0,-1)~(1,0)~(0,-1)~3(0,-2)~(0,1).
$$
\end{example}

In \cite[Theorem 3.1]{HKSS25}, for every nonnegative integers $n$ and $t$,
the map 
$$\phi : \LFp{n,t}
\to \IS_{n+1,t}(102)$$
is a bijection. 
Combining with the map $\eta$, we have the following result. 

\begin{thm}\label{thm:IS}
For every nonnegative integers $n$ and $t$,
the map
$$\phi\circ \eta^{-1} : \wHw{n,t}
\to \IS_{n+1,t}(102)$$
is a bijection. 
\end{thm}

\section{Simple $H$-paths}
\label{sec:102-000-avoiding}

In this section, we identify the weighted $H$-walks corresponding to 
\( (102,000) \)-avoiding inversion sequences and 
analyze the structure.

For each non-south step $(a,b)$ in a weighted $H$-walk $R$, we call a step $(a,b)$ 
\begin{itemize}
\item \emph{north} if $a=0$ and $b=1$,
\item \emph{up} if $a\ge1$ and $b=1$,
\item \emph{down} if $a\ge1$ and $b\le0$.
\end{itemize}
Consider weighted $H$-walks $R$ satisfying the following conditions:
\begin{itemize}
\item All south steps in $R$ are of weight $1$. 
\item A north step or south step in $R$ cannot be followed immediately by a north step.
\item A south step in $R$ cannot be followed immediately by a south step.
\end{itemize}
Note that $R\in \wHw{n}$ guarantees that a {north} step in $R$ cannot be followed immediately by a {south} step.
We refer to such weighted $H$-walks as \emph{simple $H$-paths}. 
We denote the set of simple $H$-paths $R$ of semilength \( n \) as $\Hw{n}$ and the set of $R$ in $\Hw{n}$ with $\height(R)=t$ as $\Hw{n,t}$.

\begin{lem}\label{lem:102000}
Let $\phi\circ\eta^{-1}$ be the function in Theorem~\ref{thm:IS}.
For nonnegative integers \( n \) and \( t \), we have 
$$(\phi\circ\eta^{-1})(\Hw{n,t})=\IS_{n+1,t}(102,000).$$
\end{lem}

\begin{proof} 
Let $\LFp{n,t}(000)$ be the set of $Q\in \LFp{n,t}$ satisfying the following conditions:
\begin{itemize}
\item No north steps occur consecutively. 
\item The semilength of each long step is $2$.
\item Every long step is not followed by a north step.
\end{itemize}
From the definition of $\phi$ appeared in \cite[Subsection 3.1]{HKSS25},
it is easy to prove
$$\phi(\LFp{n,t}(000))=\IS_{n+1,t}(102,000).$$
By the definition of the map $\eta$, a long step $(a,b)$ with a label $(a;b_1, b_2)$ in a labeled $F$-path $Q$ is replaced by steps $(a,b_2+1)~(0,b_1-1)$ in the weighted $H$-walk $\eta(Q)$.
Thus every south step in $\eta(Q)$ has weight $1$. Therefore 
$$\eta(\LFp{n,t}(000))=\Hw{n,t},$$  
which completes the proof.
\end{proof}
Note that $R=s_1 s_2 \cdots s_n$ is a simple $H$-path of semilength $n$ if and only if 
$R$ starts at the origin, does not go below the line $y=x$, and satisfies
\begin{itemize}
\item $s_i \in H$ ($1\le i \le n$),
\item for $1 \le i \le n-1$, $s_i s_{i+1} \neq (0,b)(0,b')$, where $b, b' \in \{m: m\le 1,~m\neq0\}$.
\end{itemize}
\begin{defn}
Let us define $\ud(R)$ of $R\in \Hw{n}$ and $\dist(e)$ of $e\in \IS_{n}(102,000)$.
\begin{itemize}
\item Given a simple $H$-path $R$, let $\ud(R)$ be the sum of the number of up steps and the number of down steps in $R$. We denote  
$\Hw{n,t}^{(m)}$ the set of simple $H$-paths of semilength $n$ with $\ud(R)=m$ {and $\height(R)=t$}.
\item 
Given an inversion sequence $e = (e_1, e_2, \dots, e_n) \in \IS_n(102,000)$, let
$$\dist(e)=\dist(e_1, e_2, \ldots, e_{n}):=\left|\{e_i : 1\le i\le n\}\right|.$$
We denote {$\IS_{n}^{(m)}(102,000)$} the set of inversion sequences $e\in\IS_{n}(102,000)$ with $\dist(e)=m$,  
{$\IS_{n,t}^{(m)}(102,000)$} the set of $e\in\IS_{n}^{(m)}(102,000)$ with {$\rank(e)=t$}.
{Note that if  $e\in\IS_{n}(102,000)$ has $\dist(e)=m$ then it holds that $m\le n\le 2m$.}
\end{itemize}
\end{defn}
Again, using the bijection $\phi\circ\eta^{-1}$ in 
Theorem~\ref{thm:IS},
we have the following result.
\begin{lem}\label{lem:102000m}
Let $\phi\circ\eta^{-1}$ be the function in Theorem~\ref{thm:IS}.
For nonnegative integers \( n \), \(m \) {and $t$}, we have 
{$$(\phi\circ\eta^{-1})(\Hw{n,t}\pex{m})=\IS_{n+1,t}\pex{m+1}(102,000).$$ }
\end{lem}

We now turn to the enumeration of simple $H$-paths using generating functions.
Let $\mathcal{A}_{n}\pex{m}$ be the set of simple $H$-paths $R$ in $\Hw{n,0}\pex{m}$ satisfying that $R$ starts with the up step $(1,1)$ and ends with an up or down step.
Let $\wtd{\mathcal{D}}_{n,t}\pex{m}$ be the set of simple $H$-paths $R$ in $\Hw{n,t}\pex{m}$ satisfying that $R$ does not start with the north step. Note that 
 $\mathcal{A}_{0}\pex{m}=\mathcal{A}_n^{(0)} = \emptyset$ 
 and $\wtd{\mathcal{D}}_{0,0}\pex{0}=\{((0,0))\}$. 
Set
\begin{align*}
A=A(x,y)&:=\sum_{n\ge 1}\sum_{m\ge 0} \left|\mathcal{A}_{n}\pex{m}\right| y^m x^n,\\
\wtd{D}_t=\wtd{D}_t(x,y)&:=\sum_{n\ge 0}\sum_{m\ge 0} \left|\wtd{\mathcal{D}}_{n,t}\pex{m}\right| y^m x^n.
\end{align*}
Decomposing simple $H$-paths in $\bigcup_{m,n\ge0}\wtd{\mathcal{D}}_{n,t}\pex{m}$ induces 
$${D}_t=(xA)^t \wtd{D}_0.$$ 
To understand the generating function $\wtd{D}_0$, let us consider the last step of $R\in \bigcup_{m\ge0}\wtd{\mathcal{D}}_{n,0}\pex{m}$ with $n\ge1$.
\begin{itemize}
\item If the last step is not a south step, then this case contributes $A$.
\item 
If the last step is a south step $(0,-k)$ for a positive integer $k$, then this case contributes 
$$A\cdot \underbrace{(x A) \cdots (x A)}_k\cdot x,$$
where the first $k$ $x$'s come from $k$ north steps and the last $x$ comes from the south step $(0,-k)$.
\end{itemize}
Thus we get
\begin{equation} \label{eq:decomp000}
\wtd{D}_0= 1+A + \sum_{k\ge1} (xA)^{k+1} 
= 1+A+\frac{x^2 A^2}{1-xA}. 
\end{equation}
Given a simple $H$-path $R\in\wtd{\mathcal{D}}_{n,t}\pex{m}$, $R(a,b)$ is an $H$-path in $\mathcal{A}_{n+1}\pex{m+1}$ if and only if 
$a-b=t$ and $1\le a\le t+1$. So we get
\begin{equation} 
{A}= \sum_{t\ge0} \wtd{D}_t \cdot (t+1)xy 
= xy\wtd{D}_0 \sum_{t\ge0}(t+1) (xA)^t 
= \frac{xy\wtd{D}_0}{(1-xA)^2}.  \label{eq:dcp000a}
\end{equation}
Substituting the equation~\eqref{eq:decomp000} for \(\wtd{D}_0 \)
into the identity in~\eqref{eq:dcp000a}, we obtain the following equation for \( A \):
\begin{equation}\label{eq:A0000}
A=\frac{xy}{(1-xA)^2}\left(1+A+\frac{x^2 A^2}{1-xA}\right).
\end{equation}
Set $B=B(x,y):=\frac{1}{1-xA}$, which is equivalent to $A=\frac{B-1}{xB}$.
Substituting $\frac{B-1}{xB}$ for \({A} \) into the equation~\eqref{eq:A0000}, we get
\begin{equation}\label{eq:B0000}
B=1+xyB^2\left(xB^2-(x-1)(B-1)\right).
\end{equation} 
We can obtain a closed form of each coefficient of $B(x,y)$ as follows. 
\begin{lem}\label{lem:form-bnm}
For nonnegative integers $n$ and $m$, let $b_{n}^{(m)}:=[x^n y^m] B(x,y)$. Then $b_0^{(m)}=\delta_{0,m}$ and
\begin{equation*} \label{eq:form-bnm}
b_{n}^{(m)}= \dfrac{1}{m}{m\choose n-m}
\sum_{j=0}^{n-m} {n-m\choose j}{n+m-j\choose 2m+j+1} \quad(n\ge 1).
\end{equation*}
\end{lem}
\begin{proof}
Set $P(u):=xu^{2}\left(xu^2-(x-1)(u-1)\right)$, then \eqref{eq:B0000} changes to
\begin{equation*}
B(x,y)=1+yP\left(B(x,y)\right). 
\end{equation*}
Applying the Lagrange inversion formula with respect to the variable $y$, 
in the form given by Deutsch~\cite{Deutsch99}, we obtain that
\begin{equation} \label{eq:eqn-LagranG}
\left[y^m\right]B(x,y)=\dfrac{1}{m}\left[u^{m-1}\right]P\left(1+u\right)^m, 
\end{equation}
for every positive integer $m$. Since
\begin{equation*}
P(1+u)^m = x^m(u+1)^{2m}\left( xu^2+(x+1)u+x\right)^m,
\end{equation*}
we have
\begin{align*}
\dfrac{1}{m}\left[u^{m-1}\right]P(1+u)^m 
&=\dfrac{x^m}{m}\sum_{i=0}^{m-1} {2m\choose m-1-i}\left[u^{i}\right]\left( xu^2+(x+1)u+x\right)^m\\
&=\dfrac{x^m}{m}\sum_{i=0}^{m-1} {2m\choose m-1-i}
\sum_{j=0}^m{m\choose j}x^j \left[u^{i-2j}\right]\left((x+1)u+x\right)^{m-j}.
\end{align*}
By \eqref{eq:eqn-LagranG} we have
\begin{equation*}
\left[y^m\right]B(x,y)=\dfrac{x^m}{m}\sum_{i=0}^{m-1} {2m\choose m-1-i}
\sum_{j=0}^m{m\choose j}x^j {m-j\choose i-2j}(x+1)^{i-2j}x^{m-i+j}.
\end{equation*}
Thus, for $n\ge 1$, we can deduce that
\begin{align}
[x^n y^m]B(x,y)
&=\left[x^n\right]\dfrac{x^m}{m}\sum_{i=0}^{m-1} {2m\choose m-1-i}\sum_{j=0}^m{m\choose j}x^j {m-j\choose i-2j}(x+1)^{i-2j}x^{m-i+j} \notag\\
&=\dfrac{1}{m}\sum_{i=0}^{m-1} {2m\choose m-1-i}\sum_{j=0}^m{m\choose j}{m-j\choose i-2j}\left[x^{n-2m+i-2j}\right](x+1)^{i-2j} \notag\\
&=\dfrac{1}{m}\sum_{i=0}^{m-1} {2m\choose m-1-i}
\sum_{j=0}^m{m\choose j}{m-j\choose i-2j}{i-2j\choose 2m-n}\notag\\
&=\dfrac{1}{m}\sum_{j=0}^{m} {m\choose j}{m-j\choose 2m-n}
\sum_{i=0}^{m-1}{2m\choose m-1-i}{n-m-j\choose m+j-i}\notag\\
&=\dfrac{1}{m}{m\choose 2m-n}
\sum_{j=0}^{n-m} {n-m\choose j}
\sum_{i=0}^{m-1}{2m\choose m+1+i}{n-m-j\choose m+j-i}\notag\\
&=\dfrac{1}{m}{m\choose n-m}
\sum_{j=0}^{n-m} {n-m\choose j}{n+m-j\choose 2m+j+1}, \notag
\end{align}
which completes the proof.
\end{proof}
Since $B(x,y)=\sum_{t \ge 0} (A(x,y) x)^t$, for $n\ge1$, $b_n\pex{m}$ is the cardinality of the set of 
{simple $H$-paths $R$ in $\bigcup_{t\geq0}\wtd{\mathcal{D}}_{n,t}\pex{m}$ satisfying that $R$ ends with a north step}. 
We denote this set by $\mathcal{B}_n\pex{m}$.

Let $b(y):=B(1,y)$.
Substituting $x$ with \({1} \) into~\eqref{eq:B0000}, we get
$b(y)=1+y\,b(y)^4$.
Thus
\begin{equation*}\label{eq:sum_bn}
\sum_{n=m}^{2m} b_n^{(m)} = [y^m]b(y)=\frac{1}{3m+1}\binom{4m}{m}.
\end{equation*}  
Applying the bijection $\phi\circ\eta^{-1}$ in Lemma~\ref{lem:102000} on $\bigcup_{n= m}^{2m}\mathcal{B}_n\pex{m}$, we get 
the following result: 
\begin{thm}
\label{cor:102000-bdistm} 
For \(m \geq 0\), the cardinality of $\bigcup_{n=m}^{2m}\mathcal{B}_n\pex{m}$, i.e., the number of $(102,000)$-inversion sequences $e$ with $\dist(e)=m$ and $|\{p:e_p=\max(e)|=1$ is given by the $3$-Fuss-Catalan number~\cite[A002293]{OEIS}
\begin{equation*}
\frac{1}{3m+1}\binom{4m}{m}.
\end{equation*}
\end{thm}

\section{Refined enumeration of $(102,000)$-avoiding inversion sequences}
\label{sec-refined-enumeration}
In this section, 
we enumerate the number of $(102,000)$-avoiding inversion sequences $e$ with respect to their length, 
number of distinct elements, and rank.
For convention, the set $\IS_{0}(102,000)$ is a singleton containing the empty sequence $\varepsilon$ with $\dist(\varepsilon)=\rank(\varepsilon)=0$. 

Let
\begin{align*}
H(x,y,z)&:=\sum_{n\ge 0}\sum_{m\ge0}\sum_{t\ge0} \left|\Hw{n,t}\pex{m}\right| z^t y^m x^n,
\\
\wtd{D}(x,y,z)&:=\sum_{n\ge 0}\sum_{m\ge0}\sum_{t\ge0} \left|\wtd{\mathcal{D}}_{n,t}\pex{m}\right| z^t y^m x^n,\\
E(x,y,z)&:=\sum_{n\ge 0}\sum_{m\ge0}\sum_{t\ge0} \left|\IS_{n,t}\pex{m}(102,000) \right| z^t y^m x^n.
\end{align*}
Decomposing simple $H$-paths induces
\begin{equation} \label{eq:wtdA-wtdA0}
\wtd{D}(x,y,z)=\sum_{t\ge0} \wtd{D}_t \,z^t = \frac{\wtd{D}_0}{1-xzA}. 
\end{equation}
Since a nonempty simple $H$-path of semilength $n$ starting with the up step $(1,1)$ corresponds to a simple $H$-path of semilength $n-1$, we have
\begin{equation*} \label{eq:wtdA-H}
\wtd{D}(x,y,z)=1+xyH(x,y,z).
\end{equation*}
Since $E(x,y,z) = 1 + xy H(x,y,z)$ by Lemma~\ref{lem:102000m}, we have 
$$E(x,y,z)=D(x,y,z).$$
From this equation, {$E(x,y,z)$} can be expressed in terms of $B=B(x,y)$ as follows: 
\begin{align}
\hspace{40mm} E(x,y,z)&={D}(x,y,z) \notag\\
&=\frac{\wtd{D}_0}{1-xzA} \hspace{60mm}\text{(by~\eqref{eq:wtdA-wtdA0})} \notag\\
&=\frac{A(1-xA)^2}{xy(1-xzA)} \hspace{54mm}\text{(by~\eqref{eq:dcp000a})}\notag\\
&=\frac{B-1}{x^2yB^2(B-z(B-1))} \hspace{36mm}\text{(by~$A=\textstyle\frac{B-1}{xB}$)}\notag\\
&=\frac{xB^2-(x-1)(B-1)}{xB-xz(B-1)}.  \hspace{40mm}\text{(by~\eqref{eq:B0000})} \label{eq:H-B}
\end{align}
From \eqref{eq:H-B} with $z=1$, 
i.e.,
\begin{equation}\label{eq:E-B}
E(x,y,1) = x^{-1}(B-1)(xB+1)+1,
\end{equation}
we arrive at the following enumerative result:
\begin{thm}\label{prop:hxy}
For positive integers \(n \) and $m$, the number of 
\((102,000)\)-avoiding inversion sequences \(e\) of length \( n \) with $\dist(e)=m$ is given by 
\begin{equation}\label{eq:102000-reform}
\left|\IS_n\pex{m}(102,000)\right|=b_{n+1}^{(m)}+\sum_{k=1}^{n}\sum_{\ell=0}^{m} b_k^{(\ell)} b_{n-k}^{(m-\ell)},
\end{equation}
where $b_n^{(m)}$ is given in Lemma~\ref{lem:form-bnm}.
\end{thm}
\begin{table}[t]
\centering
\begin{tabular}{c|cccccccccccccccccccccccccc}
\noalign{\smallskip}\noalign{\smallskip}
\(n \backslash m\)&& 1 && 2 && 3 && 4 && 5 && 6 && 7 && 8 \\ 
\hline
1 && 1 && 0 && 0 && 0 && 0 && 0 && 0 && 0 \\ 
2 && 1 && 1 && 0 && 0 && 0 && 0 && 0 && 0 \\ 
3 && 0 && 4 && 1 && 0 && 0 && 0 && 0 && 0 \\ 
4 && 0 && 4 && 9 && 1 && 0 && 0 && 0 && 0 \\ 
5 && 0 && 0 && 23 && 16 && 1 && 0  && 0 && 0 \\ 
6 && 0 && 0 && 19 && 76 && 25 && 1 && 0 && 0 \\ 
7 && 0 && 0 && 0 && 146 && 190 && 36 && 1 && 0 \\ 
8 && 0 && 0 && 0 && 101 && 630 && 400 && 49 && 1 \\ 
9 && 0 && 0 && 0 && 0 && 972 && 2010 && 749 && 64 \\ 
10 && 0 && 0 && 0 && 0 && 576 && 5160 && 5285 && 1288 \\ 
11 && 0 && 0 && 0 && 0 && 0 && 6658 && 19943 && 12124 \\ 
12 && 0 && 0 && 0 && 0 && 0 && 3445 && 41895 && 62650 \\ 
13 && 0 && 0 && 0 && 0 && 0 && 0 && 46475 && 189784 \\ 
14 && 0 && 0 && 0 && 0 && 0 && 0 && 21323 && 337876 \\ 
15 && 0 && 0 && 0 && 0 && 0 && 0 && 0 && 328786 \\ 
16 && 0 && 0 && 0 && 0 && 0 && 0 && 0 && 135439 \\ 
\hline
$[y^m] g(y)$ && 2 && 9 && 52 && 340 && 2394 && 17710 && 135720 && 1068012 \\ 
\end{tabular}
\linebreak
\caption{The number of \((102,000)\)-avoiding inversion sequences $e$ of length \( n \) with $\dist(e)=m$. (Theorem~\ref{prop:hxy})}\label{tab:ISmn}
\end{table}
Table~\ref{tab:ISmn} shows some numbers $\left|\IS_n\pex{m}(102,000)\right|$.
From \eqref{eq:B0000} and \eqref{eq:E-B}, we can deduce 
a minimal polynomial
for $E(x,y,1)$.
\begin{prop} \label{prop:H-ftneq}
The generating function 
$$E=E(x,y,1)=\sum_{n\ge 0}\sum_{m\ge0}
\left|\IS_{n}\pex{m}(102,000) \right| y^m x^n$$ 
is algebraic with minimal polynomial
\begin{equation}\label{eq:E0000}
1- \left(1-2yx(1-x)\right)E + yx\left(-1+3x+yx(1-x)^2\right)E^2  + 2y^2 x^3(1-x)E^3+y^2 x^4 E^4.
\end{equation}
\end{prop}
\begin{proof}
From~\eqref{eq:E-B}, $xB^2=(x-1)(B-1)+xE$. Thus, from~\eqref{eq:B0000},
$B-1=xyE(xE+(x-1)(B-1))$, i.e.,
$$
B=1+\frac{x^2yE^2}{1-x(x-1)yE}.
$$
Therefore it holds that
$$
x\left(1+\frac{x^2yE^2}{1-x(x-1)yE} \right)^2=(x-1)\frac{x^2yE^2}{1-x(x-1)yE}+xE.
$$
Multiplying both sides by $x^{-1}(1-x(x-1)yE)^2$, we have
$$
\left({1-x(x-1)yE}+{x^2yE^2} \right)^2=(x-1){xyE^2}{(1-x(x-1)yE)}+E(1-x(x-1)yE)^2.
$$
Rewriting this equation in ascending order with respect to $E$ gives~\eqref{eq:E0000}.
\end{proof}
From \eqref{eq:B0000} and \eqref{eq:H-B}, we can also deduce 
a minimal polynomial
for $E(x,y,z)$.
\begin{thm} \label{prop:E-ftneq}
The generating function $E=E(x,y,z)$ 
is algebraic with minimal polynomial
\begin{align*}\label{eq:ftn_Exyz}
1-&\left(1-2yx(1-x)z+y(1-x)^2 (1-z) \right) E\\
+&y\left( 2 x^2  z^2 +x (1-x) (2 z^2 - 6 z + 3) +y x^2 (1-x)^2  z^2 + (1-x)^2 (1-z) - y x(1-x)^3  (1-z)   \right) E^2 \\
+& yx\left(  2 y x^2 (1-x) z^3 +x z (1 - z) (4 - z)  - 2 y x (1-x)^2 z(1-z) - 2   (1-x) (1 - z)^2\right) E^3\\
+&yx^2\left(yx^2 z^4 - y x (1-x) z^2 (1 - z) + (1 - z)^3\right) E^4. 
\end{align*}
\end{thm}
\begin{proof}
The idea of the proof is very similar to Proposition~\ref{prop:H-ftneq}, so we omit it.
\end{proof}

Here we consider some consequences derived from Theorem~\ref{prop:E-ftneq}.
\begin{enumerate}
\item 
Substituting $y$ with \({1} \) into Theorem~\ref{prop:E-ftneq} and setting 
$$E=E(x,1,z)=\sum_{n\ge 0}\sum_{t\ge0} \left|\IS_{n,t}(102,000) \right| z^t x^n,$$ we get a minimal polynomial
\begin{align*}
1-&\left(1 -2x(1-x)z + (1-x)^2 (1 - z) \right) E\\
+&\left(2 x^2  z^2 + x (1-x) (2 z^2 - 6 z + 3)  + x^2 (1-x)^2  z^2 + (1-x)^2 (1 - z)  - x(1-x)^3  (1 - z) \right) E^2 \\
+& x\left(2 x^2 (1-x) z^3+ x z (1 - z) (4 - z) -2 x (1-x)^2 z(1 - z)- 2 (1-x) (1 - z)^2  \right) E^3\\
+&x^2\left(x^2 z^4 - x (1-x) z^2 (1 - z) +  (1 - z)^3  \right) E^4,
\end{align*} 
which is essentially the same polynomial with~\cite[p. 55]{Testart24}.
\item 
Substituting $z$ with \({1} \) into Theorem~\ref{prop:E-ftneq}  and setting 
$$E:=E(x,y,1)=\sum_{n\ge 0}\sum_{m\ge0} \left|\IS_{n}\pex{m}(102,000) \right| y^m x^n,$$ 
we can restore the polynomial in Proposition~\ref{prop:H-ftneq}.
\item 
Substituting $y$ and $z$ with \({1} \) into~Theorem~\ref{prop:E-ftneq} and setting 
$$F:=E(x,1,1)=\sum_{n\ge 0} \left|\IS_{n}(102,000) \right| x^n,$$ we get a minimal polynomial
\begin{equation*}
1 - (1-2x+2x^2)F+ x(-1+4x-2x^2+x^3)F^2 + 2x^3(1-x)F^3+x^4 F^4,
\end{equation*}
which is 
a minimal polynomial
conjectured by Testart~\cite[Conjecture 69]{Testart24}. 
It means that Theorem~\ref{prop:E-ftneq}
is a generalization of the 
minimal polynomial
given by Testart.
\item
Substituting $z$ with \({0} \) into Theorem~\ref{prop:E-ftneq}  and setting 
$$E:=E(x,y,0)=\sum_{n\ge 0}\sum_{m\ge0} \left|\IS_{n,0}\pex{m}(102,000) \right| y^m x^n,$$ we get a minimal polynomial
\begin{equation*}
1-\left(1+y(1-x)^2\right) E+y(1-x)\left(1+2x-yx(1-x)^2\right) E^2-2yx(1-x) E^3+y x^2 E^4.
\end{equation*} 
\item 
Substituting $z$ with \({0} \) and $y$ with $1$ into Theorem~\ref{prop:E-ftneq}  and setting 
$$E:=E(x,1,0)=\sum_{n\ge 0} \left|\IS_{n,0}(102,000) \right| x^n,$$ we get a minimal polynomial
\begin{equation*}
1-(2-2x+x^2) E-(1-x)(1+x+2x^2-x^3) E^2-2x(1-x) E^3+ x^2 E^4.
\end{equation*} 
\item 
Substituting $x$ with \({1} \) into~Theorem~\ref{prop:E-ftneq} and setting 
$$G:=E(1,y,z)=\sum_{m\ge0}\sum_{t\ge0} \left|\bigcup_{n=m}^{2m}\IS_{n,t}\pex{m}(102,000) \right| z^t y^m,$$ we get a minimal polynomial
\begin{equation*}
1-G+2 y z^2 G^2+yz(1-z)(4-z) G^3+y\left(y z^4+(1-z)^3\right) G^4.
\end{equation*} 
\item 
Substituting $x$ with $1$ and $z$ with \({0} \) into~Theorem~\ref{prop:E-ftneq} and setting 
$$G_0:=E(1,y,0)=\sum_{m\ge0} \left|\bigcup_{n=m}^{2m}\IS_{n,0}\pex{m}(102,000) \right| y^m,$$ we get a minimal polynomial
\begin{equation*}
1 - G_0 + y G_0^4.
\end{equation*} 
\item 
Substituting $x$ and $z$ with \({1} \) into~Theorem~\ref{prop:E-ftneq} and setting 
$$g:=E(1,y,1)=\sum_{m\ge0} \left|\bigcup_{n=m}^{2m}\IS_{n}\pex{m}(102,000) \right| y^m,$$ we get a minimal polynomial
\begin{equation*}
1 - g + 2 y g^2 + y^2 g^4 \qquad\text{(equivalently $g=(1+yg^2)^2$.)}
\end{equation*} 
\end{enumerate}

Now we calculate the coefficients of $G(y,z)=E(1,y,z)$ and $G_0(y)=E(1,y,0)$.
Recall $b(y)=B(1,y)$ satisfies $b(y)=1+yb(y)^4$. From \eqref{eq:H-B} with $x=1$, we can deduce that 
\begin{equation}\label{eq:E-yz}
G(y,z)=\frac{b(y)^2}{b(y)+(1-b(y))z}=\frac{b(y)}{1-yb(y)^3 z}.
\end{equation}
Thus for nonnegative integer $t$, 
$$\left[z^t\right]G(y,z)=y^{t} b(y)^{3t+1}.
$$
By Lagrange inversion formula \cite[Sec. 3.3]{Ges16}, 
$$\left[y^m z^t\right]G(y,z)=\left[y^{m-t}\right] b(y)^{3t+1}=\frac{3t+1}{4m-t+1}\binom{4m-t+1}{m-t}\qquad(m\ge t).$$
In particular, $[y^m z^0]G(y,z)=[y^m]G_0(y)$.
These yield the following result:
\begin{thm}\label{cor:102000-distm_rankt}
For nonnegative integer \(m \geq t\), the number of 
\((102,000)\)-avoiding inversion sequences \(e\) with $\dist(e)=m$ and $\rank(e)=t$ is
\begin{equation*}
\left|\bigcup_{n=m}^{2m}\IS_{n,t}\pex{m}(102,000)\right|=\frac{3t+1}{3m+1}\binom{4m-t}{m-t},
\end{equation*}
where the sequence appears in~\cite[A355174]{OEIS}. In particular, 
\begin{equation}\label{eq:102000-3Fuss}
\left|\bigcup_{n=m}^{2m}\IS_{n,0}\pex{m}(102,000)\right|=\frac{1}{3m+1}\binom{4m}{m}, 
\end{equation}
i.e., the number of 
\((102,000)\)-avoiding inversion sequences \(e\) with $\dist(e)=m$ and $\rank(e)=0$ is given by
the $3$-Fuss-Catalan number {\cite[A002293]{OEIS}}. 
\end{thm}
\begin{rem}
Given an $(102,000)$-avoiding inversion sequence $e=(e_1,\ldots,e_n)$ with $\dist(e)=m$ and $\rank(e)=0$, we can obtain an $(102,000)$-avoiding inversion sequence $e'$ which satisfies 
$\dist(e')=m$ and $\left|\{ p: {e}_p'=\max(e') \}\right| =1$ as follows:
\begin{enumerate}
\item If $\left|\{ p: {e}_p=\max(e) \}\right| =1$ then set $e'=e$.
\item If $\left|\{ p: {e}_p=\max(e) \}\right| =2$ with $\prmx(e)=q$ then set $e'=(e_1,\ldots, e_{q-1}, e_{q+1},\ldots,e_n)$.
\end{enumerate}
Since $\rank(e)=0$, $q=\max(e)+1$. Thus, it is easy to show that this correspondence is reversible. Therefore, a bijective proof of~\eqref{eq:102000-3Fuss} 
induces a bijective proof of Theorem~\ref{cor:102000-bdistm}.
\end{rem}
We can also calculate the coefficients of $g(y)=E(1,y,1)$.
Substituting $z$ with \({1} \) into~\eqref{eq:E-yz} (or $x$ with $1$ into~\eqref{eq:E-B}), we get
\begin{equation*}\label{eq:fxy}
g(y)=b(y)^2,
\end{equation*}
which implies that 
$$\left[y^m\right]g(y)=\left[y^m\right]b(y)^2=\frac{2}{4m+2}\binom{4m+2}{m}.$$ 
This yields the following result:
\begin{cor}\label{cor:102000-distm}
For \(m \geq 0\), the number of 
\((102,000)\)-avoiding inversion sequences \(e\) with $\dist(e)=m$ is
\begin{equation*}
\left|\bigcup_{n=m}^{2m}\IS_{n}\pex{m}(102,000)\right|=\frac{1}{2m+1}\binom{4m+2}{m},
\end{equation*}
where the sequence appears in~\cite[A069271]{OEIS}.
\end{cor}
The last row of Table~\ref{tab:ISmn} shows numbers $\left|\bigcup_{n=m}^{2m}\IS_n\pex{m}(102,000)\right|$ for $m=1,\ldots ,8$.
\begin{rem}
The following problems are left for future research.
\begin{enumerate}
\item Find closed forms for $\left|\IS_{n,t}\pex{m}(102,000)\right|$ and $\left|\IS_{n,t}(102,000)\right|$.
\item Find a simpler formula for
~\eqref{eq:102000-reform}.
\item Find bijective proofs of Theorem~\ref{cor:102000-distm_rankt} and Corollary~\ref{cor:102000-distm}.
\end{enumerate}
\end{rem}

\end{document}